\documentclass[12pt]{amsart}
\usepackage{fullpage}
\usepackage[utf8]{inputenc}
\usepackage[english]{babel}
\usepackage{mathtools}
\usepackage{amssymb}
\usepackage{amsthm}
\usepackage{mathrsfs}
\usepackage{nicefrac}
\usepackage{upgreek}
\usepackage[T1]{fontenc}
\usepackage{graphicx}
\usepackage{hyperref}

\usepackage[all]{xy}
\usepackage{graphics}
\usepackage{color}
\usepackage[T1]{fontenc}
\usepackage{parskip}

\usepackage{enumitem}

\makeatletter
\newtheorem*{rep@theorem}{\rep@title}
\newcommand{\newreptheorem}[2]{%
\newenvironment{rep#1}[1]{%
 \def\rep@title{#2 \ref{##1}}%
 \begin{rep@theorem}}%
 {\end{rep@theorem}}}
\makeatother

\newtheorem{theorem}{Theorem}[section]
\newtheorem{lemma}[theorem]{Lemma}

\newtheorem{conjecture}[theorem]{Conjecture}
\theoremstyle{definition}\newtheorem{definition}[theorem]{Definition}
\theoremstyle{remark}\newtheorem{remark}[theorem]{Remark}
\theoremstyle{definition}
\makeatletter\@addtoreset{case}{example}\makeatother
\theoremstyle{definition}

\begin{document}

\title{The dichotomy property of ${\rm SL}_2(R)$-A short note}

\author{Alexander A. Trost}
\address{Fakult\"{a}t f\"{u}r Mathematik, Ruhr Universit\"{a}t Bochum, D-44780 Bochum, Germany}
\email{Alexander.Trost@ruhr-uni-bochum.de}

\begin{abstract}
A recent paper by Polterovich, Shalom and Shem-Tov has shown that non-discrete, conjugation invariant norms on arithmetic Chevalley groups of higher rank give rise to very restricted topologies. Namely, such topologies always have profinite norm-completions. In this note, we sketch an argument showing that this also holds for ${\rm SL}_2(R)$ for $R$ a ring of algebraic integers with infinitely many units. 
\end{abstract}

\maketitle

\section{Introduction} 
\label{intro}

The preprint \cite{polterovich2021norm} by Polterovich, Shalom and Shem-Tov studies (among other things) the topologies induced by conjugation-invariant norms on ${\rm SL}_n(R)$ for $n\geq 3$ and rings $R$ for which ${\rm SL}_n(R)$ has a certain finiteness property, namely that there is a natural number $L(R,n)$ such that each element of ${\rm SL}_n(R)$ can be written as a product of at most $L(R,n)$ elementary matrices. This property is called \textit{bounded elementary generation.} They show in this case that the following theorem holds:

\begin{theorem}\cite[Theorem~1.8]{polterovich2021norm}\label{old_theorem}
Let $R$ be an integral domain, $n\geq 3$ and assume that ${\rm SL}_n(R)$ is boundedly elementary generated and that each non-zero ideal of $R$ has finite index. Then each conjugation-invariant norm on a finite index subgroup of ${\rm SL}_n(R)$ is either discrete or has a profinite norm-completion.
\end{theorem}

A group satisfying the conclusion of the theorem is said to satisfy the \textit{dichotomy property}. In any case, there are quite many rings of interest that satisfy the assumption of the theorem, for example rings of integers in global and local fields. The assumption of $n\geq 3$ in Theorem~\ref{old_theorem} is also quite common in the study of conjugation-invariant norms on Chevalley groups. This is mostly due to the fact that the normal subgroup structure of ${\rm SL}_2(R)$ is more complicated than the one of ${\rm SL}_n(R)$ for $n\geq 3.$ Namely, the normal subgroups of the latter are essentially parametrized by ideals of the ring $R$ (or more precisely so-called \textit{admissible pairs of ideals} as shown by Abe \cite{MR991973}), whereas normal subgroups of ${\rm SL}_2(R)$ are given by more complicated subgroups of $(R,+)$ called \textit{radices} as proven by Costa and Keller \cite{MR1114610}. However, as seen in our preprint \cite{SL_2_strong_bound}, it is often possible to generalize results of the study of conjugation invariant norms from $n\geq 3$ to $n=2$ by introducing additional assumptions on the existence of a lot of units in the ring. In this context, \cite{polterovich2021norm} raises the following question:

\begin{conjecture}\label{conj}
Let $p$ be a prime in $\mathbb{Z}.$ Does ${\rm SL}_2(\mathbb{Z}[1/p])$ satisfy the dichotomy property?
\end{conjecture}

We will answer the question in Conjecture~\ref{conj} affirmatively with the following more general result:  


\begin{theorem}\label{main_theorem}
Let $R$ be the ring of S-algebraic integers in a number field such that $R$ has infinitely many units. Then every finite index subgroup of ${\rm SL}_2(R)$ has the dichotomy property.  
\end{theorem}

The proof is almost identical to the proof of \cite[Theorem~1.8]{polterovich2021norm} itself and only requires a small modification in a technical lemma. 


\section*{Acknowledgments}

This note was written during a research visit at the Mathematische Forschungsinstitut Oberwolfach and I am very grateful for the support I received there. I also want to thank Zvi Shem-Tov for his helpful remarks.

\section{Basic definitions and notions}
\label{sec_basic_notions}

\begin{definition}
Let $G$ be a group with neutral element $1_G$. A \textit{conjugation-invariant norm} $\|\cdot\|:G\to\mathbb{R}_{\geq 0}$ is a function satisfying the properties
\begin{align*}
&\forall a\in G:\|a\|=0\iff a=1_G,\\
&\forall a\in G:\|a\|=\|a^{-1}\|,\\
&\forall a,b\in G:\|ab\|\leq\|a\|+\|b\|\text{ and}\\
&\forall a,b\in G:\|aba^{-1}\|=\|a\|
\end{align*}
\end{definition}

Further, we recall the following two concepts from \cite{polterovich2021norm}:

\begin{definition} 
A group $G$ is said to satisfy the \textit{dichotomy property}, if each non-discrete norm $\|\cdot\|$ on $G$ has a profinite norm-completion. If $G$ is itself a topological group, then $G$ is called \textit{norm-complete}, if each non-discrete norm $\|\cdot\|$ on $G$ induces the topology of $G$ as a topological group. 
\end{definition}

Let us next recall the definition of ${\rm SL}_2:$

\begin{definition}
Let $R$ be a commutative ring with $1$. Then ${\rm SL}_2(R):=\{A\in R^{2\times 2}\mid a_{11}a_{22}-a_{12}a_{21}=1\}.$
\end{definition}

Obviously for any commutative ring $R$ and any $x\in R,$ the matrices 
\begin{equation*}
E_{12}(x)=
\begin{pmatrix}
1 & x\\
0 & 1
\end{pmatrix}
\text{ and }
E_{21}(x)=
\begin{pmatrix}
1 & 0\\
x & 1
\end{pmatrix}
\end{equation*}
are elements of ${\rm SL}_2(R).$ They are called \textit{elementary matrices} and we denote the set of elementary matrices $\{E_{12}(x),E_{21}(x)\mid x\in R\}$ by ${\rm EL}.$ The subgroup of ${\rm SL}_2(R)$ generated by ${\rm EL}$ is denoted by $E(2,R).$ Furthermore, for an ideal $I\unlhd R,$ we denote by $E(2,R,I)$ the normal subgroup of $E(2,R)$ generated by the $E(2,R)$-conjugates of elements of $\{E_{12}(x)\mid x\in I\}.$ Additionally, for an ideal $I\subset R$, the subgroups $E_{12}(I)$ and $E_{21}(I)$ of $E(2,R)$ are defined as $E_{12}(I):=\{E_{12}(x)\mid x\in I\}$ and $E_{21}(I):=\{E_{21}(x)\mid x\in I\}.$ Further, for a unit $u\in R^*$ the element 
\begin{equation*}
h(u)=
\begin{pmatrix}
u & 0\\
0 & u^{-1}
\end{pmatrix}
\end{equation*}
is also an element of ${\rm SL}_2(R).$ 



\section{Proof of the main result}
\label{proof_main}

We first introduce the needed concept of $R$ containing a large number of units:

\begin{definition}\label{many_units}
Let $R$ be a commutative ring with $1$ such that for each $c\in R-\{0\},$ there is a unit $u\in R$ such that $u-1\in c^2R$ and $u^8-1\neq 0.$ Then we call $R$ a \textit{ring with many units.}
\end{definition}

\begin{remark}
Note, that this is a similar but still slightly different property to the concept of rings with many units introduced in \cite{SL_2_strong_bound}.
\end{remark}

This enables us to formulate:

\begin{theorem}\label{main_thm_tech_version}
Let $R$ be an integral domain with many units such that $E(2,R)$ is boundedly generated by elementary matrices and such that each non-zero ideal of $R$ has finite index in $R$. Then each finite index subgroup $H$ of $E(2,R)$ has the dichotomy property.
\end{theorem}

To prove this, we need the following modified version of \cite[Lemma~2.3]{polterovich2021norm}:

\begin{lemma}\label{tech_lemma_1}
Let $R$ be as in Theorem~\ref{main_thm_tech_version}. Let $\|\cdot\|$ further be the restriction of a non-discrete conjugation-invariant norm on $E(2,R,I)$ to the elementary subgroup $E_{12}(I).$ Then the norm completion $\overline{E_{12}(I)}$ of $E_{12}(I)$ with respect to $\|\cdot\|$ is profinite.
\end{lemma}

To prove this we need the following technical lemma, which in turn is a modified version of \cite[Lemma~A9]{polterovich2021norm}:

\begin{lemma}\label{tech_lemma_2}
Let $R$ be an integral domain, $c\in R$ non-zero and $u\in R$ a unit such that $u-1\in c^2R$. Assume further that 
\begin{equation*}
A=
\begin{pmatrix}
a & b\\
c & d
\end{pmatrix}
\end{equation*}
is an element of ${\rm SL}_2(R).$ Then each element of $E_{12}((u^8-1)cR)$ is a product of at most four $E(2,R,cR)$-conjugates of $A$ and $A^{-1}.$ 
\end{lemma}

\begin{proof}
As $u-1$ is an element of the ideal $cR$ by assumption, so is $u^4-1.$ Hence choose $x\in R$ with $u^4-1=cx$ and set $t:=ax.$ But note that as $u-1\in c^2R$, there is an $y\in R$ such that $u-1=c^2y$ holds. Hence
\begin{equation*}
cx=u^4-1=(u-1)\cdot(u^3+u^2+u+1)=c^2y\cdot(u^3+u^2+u+1)
\end{equation*}
and so $x=cy\cdot(u^3+u^2+u+1)\in cR.$ Thus $t$ is an element of $cR.$ Then the matrix $Y:=E_{12}(t)A^{-1}E_{12}(-t)h(u^2)Ah(u^{-2})$ is a product of two $E(2,R,cR)$-conjugates of $A$ and $A^{-1}$ and has the form  
\begin{equation*}
Y=
\begin{pmatrix}
u^{-4} & q\\
0 & u^4
\end{pmatrix}
\end{equation*}
for some $q\in cR.$ But then note for $p\in R$ that 
\begin{equation*}
E_{12}((u^{-4}-u^4)(p+q))=Y\cdot
\begin{pmatrix}
u^4 & p\\
0 & u^{-4}
\end{pmatrix}
\cdot Y^{-1}
\cdot
\begin{pmatrix}
u^{-4} & -p\\
0 & u^4
\end{pmatrix}
\end{equation*}
Hence choosing $p:=-q+z$ for $z\in cR$, we obtain that $E_{12}((u^4-u^{-4})z)$ is a product of four $E(2,R,cR)$-conjugates of $A$ and $A^{-1}.$
\end{proof}

\begin{remark}
Implicit in this proof is the claim that $h(u)$ is an element of $E(2,R,cR)$. We will not prove this, but it follows from a short calculation with the standard decomposition of $h(u)$ into elementary matrices.
\end{remark}

Using this lemma, we can prove Lemma~\ref{tech_lemma_1}:  

\begin{proof}
For each ideal $J$ of $R$ contained in $I$ consider the closure $U_J$ of $E_{12}(J)$ in $\overline{E_{12}(I)}$. As a non-zero ideal $J$ in $R$ has finite index, this implies that the closed subgroup $U_J$ has finite index in $\overline{E_{12}(I)}.$ Thus $U_J$ is also open in $\overline{E_{12}(I)}.$ Hence to show that $\overline{E_{12}(I)}$ is profinite, it suffices to show that the identity in $\overline{E_{12}(I)}$ has a neighborhood basis of subgroups of the form $U_J.$ First, note that $\|\cdot\|$ is the restriction of a non-discrete norm on $E(2,R,I),$ which we will also denote by $\|\cdot\|.$ Thus for each $\epsilon>0,$ there is an $A\in E(2,R,I)$ with $\|A\|\leq\epsilon/4.$ As there are only finitely many scalar matrices in $E(2,R)$, we may by possibly choosing a smaller $\epsilon$ or by considering a suitable conjugate or commutator of $A$ assume that $A$ has the $(2,1)$-entry $c$ with $c\neq 0.$ But $R$ is a ring with many units, so we may choose a unit $u$ in $R$ with $u^8-1\neq 0$ and $u\equiv 1\text{ mod }c^2R.$ According to Lemma~\ref{tech_lemma_2}, the subgroup $E_{12}((u^8-1)cR)$ is then contained in the $4\cdot\|A\|$ ball around $E_2$ as $c$ is an element of $I.$ Hence $J_{\epsilon}:=(u^8-1)cR$ is contained in $I$ and $E_{12}(J_{\epsilon})$ is contained in the $\epsilon$-ball around $E_2.$ But this implies that $U_{J_{\epsilon}}$ is a subgroup of $\overline{E_{12}(I)}$ contained in the $\epsilon$-ball around $E_2.$   
\end{proof}                                                              

One can now prove Theorem~\ref{main_thm_tech_version} in essentially the same way as \cite[Theorem~1.8]{polterovich2021norm}:

\begin{proof}
We first prove the dichotomy property in the case of $H=E(2,R)$. So let $\|\cdot\|$ be a non-discrete, conjugation-invariant norm on $E(2,R)$ and let $G$ be the norm-completion of $E(2,R)$ with respect to $\|\cdot\|.$ Further, let $U_1$ and $U_2$ be the closures of $E_{12}(R)$ and $E_{21}(R)$ within $G$ respectively. Then by assumption the group $E(2,R)$ is boundedly generated by its elementary subgroups. Thus $G$ is boundedly generated by its subgroups $U_1,U_2$ and so there is a natural number $N(R)$ such that $G=U_{i_1}\cdot U_{i_2}\cdots U_{i_{N(R)}}$ holds for $i_1,\dots,i_{N(R)}\in\{1,2\}$. But according to Lemma~\ref{tech_lemma_1} the subgroups $U_1,U_2$ of $G$ are profinite. But it was observed in the proof of \cite[Theorem~1.8(i)]{polterovich2021norm}, that if a topological group $G$ is a set-theoretic product of finitely many subgroups profinite in the relative topology, then $G$ itself is profinite. This finishes the proof for $E_2(R)$. The general case works the same way as in \cite{polterovich2021norm} as well: Let $H$ be a finite index subgroup of $E(2,R)$. Then following precisely the line of arguments from \cite{polterovich2021norm} and the already shown case $H=E(2,R)$, one reduces the dichotomy proof for $H$ to the claim that for a non-zero ideal $I$ and a conjugation-invariant, non-discrete norm $\|\cdot\|$ on $E(2,R,I)$, said norm restricts to norms on $E_{12}(I)$ (and $E_{21}(I)$) having profinite norm-completions. But this is precisely the claim of Lemma~\ref{tech_lemma_1}. 
\end{proof}

If $R$ is an integral domain containing a unit $v$ of infinite order and each of its non-zero ideals has finite index, then it has many units: Namely, let $c\in R-\{0\}$ be given. Assuming wlog that $c$ is not a unit in $R$, we note that $R/c^2R$ is finite. Then consider $\overline{v}:=v+c^2R$. This element must have finite order in $(R/c^2R)^*$ and so there is a $k>0$ such that $\overline{v}^k=1+c^2R.$ Note that $u:=v^k$ has infinite order and so $u^8$ can not be $1.$ Also $u-1\in c^2R$ by choice of $u.$ This argument applies for example for rings of S-algebraic integers with infinitely many units as they have units of infinite order according to Dirichlet's Theorem \cite[Corollary~11.7]{MR1697859}. Additionally ${\rm SL}_2(R)$ has bounded elementary generation \cite[Theorem~1.1]{MR3892969}. Thus Theorem~\ref{main_thm_tech_version} implies Theorem~\ref{main_theorem}. Additionally, we note the following version of \cite[Theorem~1.12]{polterovich2021norm}:

\begin{theorem}
Let $R$ be a compact metrizable ring with many units such that each of its non-zero ideals has finite index and such that $R$ is also an integral domain. Then ${\rm SL}_2(R)$ equipped with the relative topology from $R^{2\times 2}$ as well as any of its finite index subgroups are norm-complete.
\end{theorem}

The proof is virtually identical to the proof of Theorem~\ref{main_thm_tech_version} above, so we will skip it. We do however want to note, that this yields a proof of \cite[Theorem~3.6]{polterovich2021norm}, which does not require the rather technical \cite[Lemma~3.7]{polterovich2021norm}.

\section{Closing remarks}

To round out this short note, we note that one can prove a generalization of \cite[Theorem~1.8]{polterovich2021norm} also for all other split Chevalley groups besides the ${\rm SL}_n$. The main difference is not the statement, which would be the same, but that a bit of additional care has to be taken in the proof: Namely, the root subgroups associated to short and long roots require slightly different approaches to prove the appropriate form of Lemma~\ref{tech_lemma_1}. Also the exceptional groups ${\rm Sp}_4$ and $G_2$ will require different strategies in case the ring $R$ in question has the bad characteristics $2$ or $3$. Ultimately though, all the split cases are very similar; the more interesting cases are likely those arising from more complicated algebraic groups.  



\bibliography{bibliography}
\bibliographystyle{plain}

\end{document}